\documentclass[11pt,reqno,tbtags]{amsart}
\usepackage{amssymb}

\usepackage[colorlinks=true, pdfstartview=FitV, linkcolor=blue,
  citecolor=blue, urlcolor=blue,pagebackref=false]{hyperref}
\usepackage[numeric,initials,nobysame]{amsrefs}
\usepackage{amsmath,amssymb,amsfonts,dsfont}
\usepackage{enumerate,graphicx}

\numberwithin{equation}{section}

\newtheorem{maintheorem}{Theorem}

\newtheorem{theorem}{Theorem}[section]
\newtheorem*{theorem*}{Theorem}
\newtheorem{lemma}[theorem]{Lemma}
\newtheorem{claim}[theorem]{Claim}

\newtheorem{corollary}[theorem]{Corollary}

\theoremstyle{definition}{

}

\theoremstyle{remark}{

\newtheorem*{remark*}{Remark}

}

\newcommand{\R}{\mathbb R}
\newcommand{\N}{\mathbb N}

\newcommand{\E}{\mathbb{E}}
\renewcommand{\P}{\mathbb{P}}
\DeclareMathOperator{\var}{Var} \DeclareMathOperator{\Cov}{Cov}

\renewcommand{\epsilon}{\varepsilon}

\newcommand{\one}{\boldsymbol{1}}

\newcommand{\deq}{\stackrel{\scriptscriptstyle\triangle}{=}}
\newcommand{\tmix}{t_{{\rm mix}}}

\newcommand{\gap}{\text{\tt{gap}}}

\newcommand{\cS}{\mathcal{S}}

\begin{document}
\title{Mixing time for the Ising model: \\ a uniform lower bound for all graphs}
\author{Jian Ding and Yuval Peres}

\address{Jian Ding\hfill\break
Department of Statistics\\
UC Berkeley\\
Berkeley, CA 94720, USA.}
\email{jding@stat.berkeley.edu}
\urladdr{}

\address{Yuval Peres\hfill\break
Microsoft Research\\
One Microsoft Way\\
Redmond, WA 98052-6399, USA.}
\email{peres@microsoft.com}
\urladdr{}

\begin{abstract}\thanks{Recently we found a simple proof which gives a $n\log n/2$ lower bound. See appendix.}
Consider Glauber dynamics for the Ising model on a graph of $n$ vertices. Hayes and Sinclair showed that the mixing time for this dynamics is at least $n\log n/f(\Delta)$, where $\Delta$ is the maximum degree and $f(\Delta) = \Theta(\Delta \log^2 \Delta)$. Their result applies to more general spin systems, and in that generality, they showed that some dependence on $\Delta$ is necessary. In this paper, we focus on the ferromagnetic Ising model and prove that
the mixing time of Glauber dynamics on any $n$-vertex graph is at least $(1/4+o(1))n \log n$.

\end{abstract}
\maketitle
\section{Introduction}\label{sec:1}

Consider a finite graph $G = (V, E)$ and a  finite alphabet $Q$. A general {\em spin system} on $G$ is a probability measure $\mu$ on $Q^V$; well studied examples in computer science and statistical physics include the uniform
 measure on proper colorings and the Ising model.
Glauber (heat-bath) dynamics are often used to sample from $\mu$
(see, e.g., \cite{Sinclair, Martinelli, LPW}). In discrete-time
Glauber dynamics, at each step a vertex $v$ is chosen uniformly at
random and the label at $v$ is replaced by a new label chosen from
the $\mu$-conditional distribution given the labels on the other
vertices. This Markov chain has stationary distribution $\mu$, and
the key quantity to analyze is the mixing time $\tmix$, at which the
distribution of the chain is close in total variation to $\mu$
(precise definitions are given below).

If $|V| =n$, it takes $(1+o(1))n\log n$ steps to update all vertices (coupon collecting), and it is natural to guess that this is a lower bound for the mixing time. However, for the Ising model at infinite temperature
or equivalently, for the 2-colorings of the graph $(V,\emptyset)$, the mixing time of Glauber dynamics is asymptotic to $n\log n/2$, since these models reduce to the lazy random walk on the hypercube, first analyzed in~\cite{Aldous}. Thus mixing can occur before all sites are updated, so the coupon collecting argument does not suffice to obtain a lower bound for the mixing time. The first general bound of the right order was obtained by Hayes and Sinclair \cite{HS}, who showed that the mixing time for Glauber dynamics is at least $n\log n/f(\Delta)$, where $\Delta$ is the maximum degree and $f(\Delta) = \Theta(\Delta \log^2 \Delta)$. Their result applies for quite general spin systems, and they gave examples of spin systems $\mu$ where some dependence on $\Delta$ is necessary. After the work of \cite{HS}, it remained unclear whether a uniform lower bound of order $n \log n$, {\em that does not depend on $\Delta$\/}, holds for the most extensively studied spin systems, such as proper colorings and the Ising model.

In this paper, we focus on the ferromagnetic Ising model, and obtain a lower bound of $(1/4+o(1))n\log n$ on any graph with general (non-negative) interaction strengths.
\bigskip

\noindent{\bf Definitions.} The \emph{Ising Model} on a finite graph
$G=(V,E)$ with interaction strengths $J=\{J_{uv} \geq 0: uv\in E\}$
 is a probability measure $\mu_G$ on the configuration space $\Omega = \{\pm1\}^V$,
defined as follows. For each $\sigma \in \Omega$,
 \begin{equation} \label{gibbs}
 \mu_G(\sigma)  = \frac1{Z(J)} \exp\Big( \sum_{uv\in E}J_{uv}\sigma(u)\sigma(v) \Big)\,,
 \end{equation}
where $Z(J)$ is a normalizing constant called the {\em partition function}.
The measure $\mu_G$ is also called the {\em Gibbs measure\/} corresponding to the interaction matrix $J$.
 When there is no ambiguity regarding the base graph, we sometimes write $\mu$ for $\mu_G$.

Recall the definition of the Glauber dynamics: At each step, a
vertex is chosen uniformly at random, and its spin is updated
according to the conditional Gibbs measure given the  spins of all the
other vertices. It is easy to verify that this chain is reversible
with respect to $\mu_G$.

Next we define the \emph{mixing time}.
Let $(X_t)$ denote an aperiodic irreducible Markov chain on a finite state space $\Omega$ with transition kernel $P$
and stationary measure $\pi$. For any two distributions $\mu,\nu$ on $\Omega$, their \emph{total-variation distance} is defined to be
$$\|\mu-\nu\|_\mathrm{TV} \deq \sup_{A \subset\Omega} \left|\mu(A) - \nu(A)\right| = \tfrac{1}{2}\sum_{x\in\Omega} |\mu(x)-\nu(x)|~.$$
 For $x\in \Omega$ let $\P_x$ denotes the probability given $X_0=x$ and let
 $$\tmix^x = \min\left\{t : \| \P_x(X_t \in \cdot)- \pi\|_\mathrm{TV} \leq \frac{1}{4} \right\}$$ be
the mixing time with initial state $x$. (The choice of $1/4$ here is
by convention, and can be replaced by any constant in $(0,1/2)$,
 without affecting the $(1/4+o(1)) n \log n$ lower bound in the next theorem.)
  The {\em mixing time} $\tmix$ is then defined to be
 $\max_{x \in \Omega} \tmix^x$.

We now state our main result.
\begin{maintheorem}\label{mainthm-1}
Consider the Ising model (\ref{gibbs})
on the graph $G$ with interaction matrix $J$,  and let $\tmix^+(G,J)$ denote the mixing time of the
corresponding Glauber dynamics, started from the all-plus
configuration. Then
\[ \inf_{G,J} \tmix^+(G,J) \geq (1/4+o(1)) n\log n\,,\]
where the infimum is over all $n$-vertex graphs $G$ and all nonnegative interaction matrices $J$.
\end{maintheorem}
\noindent{\bf Remark.} Theorem \ref{mainthm-1} is sharp up to a factor of 2.
 We conjecture that $(1/4+o(1))$ in the theorem could be replaced by
$(1/2+o(1))$, i.e., the mixing time is minimized  (at least asymptotically) by taking $J\equiv 0$.
\medskip

Hayes and Sinclair \cite{HS}
constructed spin systems where the mixing time of the
Glauber dynamics has an upper bound $O(n\log n/\log \Delta)$. This,
in turn, implies that in order to establish a lower bound of order $n \log n$ for the Ising model on a general graph,
we have to employ some specific properties of the model.
In our proof of Theorem~\ref{mainthm-1}, given in the next section, we use the GHS
inequality \cite{GHS} (see also \cite{Lebowitz} and \cite{EM}) and a recent censoring inequality \cite{PW} due to Peter Winkler and the second author.

\section{Proof of Theorem~\ref{mainthm-1}}
The intuition for the proof is the following: In the case of strong
interactions, the spins are highly correlated and the
mixing should be quite slow; In the case of weak interaction
strengths, the spins should be weakly dependent and close to the
case of the graph with no edges, therefore one may extend the arguments for
the lazy walk on the hypercube.

We separate the two cases by considering the
\emph{spectral gap}. Recall that the spectral gap of a reversible discrete-time
Markov chain, denoted by $\gap$, is $1-\lambda$, where $\lambda$ is the
second largest eigenvalue of the transition kernel. The following
simple lemma gives a lower bound on  $\tmix^+$ in
terms of the spectral gap.
\begin{lemma} \label{gap}
The Glauber dynamics for the ferromagnetic Ising model (\ref{gibbs}) satisfies $\tmix^+ \geq \log 2
\cdot (\gap^{-1} - 1)$.
\end{lemma}
\begin{proof}
It is well known that $\tmix \geq \log2 \cdot (\gap^{-1}-1)$ (see,
e.g., Theorem~12.4 in \cite{LPW}). Actually, it is shown in the proof of
\cite{LPW}*{Theorem 12.4} that $\tmix^x \geq \log 2\cdot(\gap^{-1} -
1)$ for any state $x$ satisfying $f(x) = \|f\|_\infty$, where $f$ is an
eigenfunction corresponding to the second largest eigenvalue. Since
  the second eigenvalue of
the Glauber dynamics for the ferromagnetic Ising model has an increasing
eigenfunction $f$ (see \cite{Nacu}*{Lemma 3}), we infer that either $\|f\|_\infty = f(+)$ or
$\|f\|_\infty = f(-)$. By symmetry of the all-plus and the all-minus
configurations in the Ising model (\ref{gibbs}), we have $\tmix^+=\tmix^-$, and this
concludes the proof.
\end{proof}

Lemma \ref{gap} implies that  Theorem~\ref{mainthm-1} holds if $\gap^{-1} \geq n \log n$. It remains to
 consider the case  $\gap^{-1} \leq n \log n$.

\begin{lemma}\label{lem-F}
Suppose that the Glauber dynamics for the Ising model on a graph $G= (V, E)$ with $n$ vertices satisfies $\gap^{-1} \leq  n \log n$. Then there exists a subset $F \subset V$ of size $\lfloor\sqrt{n}/\log n\rfloor$ such that
\[\sum_{u, v\in F, u\neq v} \Cov_\mu(\sigma(u), \sigma(v)) \leq \frac{2}{\log n}\,.\]
\end{lemma}
\begin{proof}
We first establish an upper bound on the variance of the sum of
spins $S = S(\sigma) = \sum_{v\in V} \sigma(v)$. The variational
principle for the spectral gap of a reversible Markov chain with
stationary measure $\pi$ gives (see, e.g., \cite{AF}*{Chapter 3} or \cite{LPW}*{Lemma 13.12}:
\begin{align*}
\gap = \inf_f \frac{\mathcal{E}(f)}
{\var_{\pi}(f)}\,,
\end{align*}
where $\mathcal{E}(f)$ is the Dirichlet form defined by
\begin{align*}
\mathcal{E}(f) = \left<(I-P)f,f\right>_{\pi} = \tfrac{1}{2}\sum_{x,y\in\Omega}\left[f(x)-f(y)\right]^2\pi(x)P(x,y)\,.
\end{align*}

Applying the variational principle with the test function $S$, we
deduce that
\begin{equation*}
\gap \leq \frac{\mathcal{E}(S)}{\var_\mu(S)}\,.
\end{equation*}
Since the Glauber dynamics updates a single spin at each step, $\mathcal{E}(S)\leq 2$, whence
\begin{equation}\label{eq-var-bound}
\var_\mu(S) \leq \mathcal{E}(S) \gap^{-1} \leq 2\,n \log n\,.
\end{equation}
The covariance of the spins for the
ferromagnetic Ising model is non-negative by the FKG inequality (see, e.g., \cite{GHM}).  Applying
Claim~\ref{claim-matrix} below with $k =
\lfloor\tfrac{\sqrt{n}}{\log n}\rfloor$ to the covariance matrix of
$\sigma$  concludes the proof of the lemma.
\end{proof}
\begin{claim}\label{claim-matrix}
Let $A$ be an $n \times n$ matrix with non-negative entries. Then
for any $k\leq n$ there exists $F\subset \{1, \ldots, n\}$ such that
$|F| = k$ and
\[\sum_{i, j \in F}A_{i, j}\one_{\{i\neq j\}} \leq \frac{k^2
}{n^2}\sum_{i\neq j} A_{i, j}\,.\]
\end{claim}
\begin{proof}
Let $R$ be a uniform random subset of $\{1, \ldots, n\}$ with $|R| =
k$. Then,
\begin{align*}
\E \Big[\sum_{i,j\in R} A_{i, j}\one_{\{i \neq j\}} \Big]& =  \sum_{1 \leq i,j \leq n} A_{i, j} \one_{\{i\neq j\}} \P( i, j \in R) \\
&= \frac{k(k-1)}{n(n-1)}  \sum_{1 \leq i,j \leq n} A_{i, j}
\one_{\{i\neq j\}} \leq \frac{k^2}{n^2}\sum_{i\neq j} A_{i, j}\,.
\end{align*}
Existence of the desired subset $F$ follows immediately.
\end{proof}
We now consider a version of accelerated dynamics $(X_t)$ with respect to the subset $F$ as in Lemma~\ref{lem-F}. The accelerated dynamics selects a vertex $v \in V$ uniformly at random  at each time and updates in the following way:
\begin{itemize}
\item If $v\not \in F$, we update   $\sigma(v)$ as in the usual
Glauber dynamics.
\item If $v \in F$, we update the spins on $\{v\} \cup F^c$
all together as a block, according to the conditional Gibbs measure given the spins on $F\setminus\{v\}$.
\end{itemize}
The next censoring inequality of \cite{PW} guarantees that, starting from the all-plus configuration, the accelerated dynamics indeed mixes faster than the original one. In what follows, write $\mu \preceq \nu$ if $\nu$ stochastically dominates $\mu$. 
\begin{theorem}[\cite{PW} and also see \cite{Peres}*{Theorem 16.5}]\label{thm-PW-speed-up}
Let $(\Omega, S, V, \pi)$ be a monotone system and let $\mu$ be the
distribution on $\Omega$ which results from successive updates at
sites $v_1, \ldots, v_m$, beginning at the top configuration. Define
$\nu$ similarly but with updates only at a subsequence $v_{i_1},
\ldots, v_{i_k}$. Then $\mu \preceq \nu$, and $\|\mu - \pi
\|_{\mathrm{TV}} \leq \|\nu-\pi\|_{\mathrm{TV}}$. Moreover, this
also holds if the sequence $v_1, \ldots, v_m$ and the subsequence
$i_1, \ldots, i_k$ are chosen at random according to any prescribed
distribution.
\end{theorem}
In order to see how the above theorem indeed implies that the
accelerated dynamics $(X_t)$ mixes at least as fast as the
usual dynamics, first note that any vertex $u\notin F$ is updated
according to the original rule of the Glauber dynamics. Second, for
$u\in F$, instead of updating the block $\{u\} \cup F^c$, we can
simulate this procedure by performing sufficiently many single-site
updates in $\{u\} \cup F^c$. This approximates the accelerated
dynamics arbitrarily well, and contains a superset of the
single-site updates of the usual Glauber dynamics. Theorem
\ref{thm-PW-speed-up} thus completes this argument.

Let  $(Y_t)$ be the projection of the chain $(X_t)$ onto the
subgraph $F$. Recalling the definition of the accelerated dynamics,
we see that $(Y_t)$ is also a Markov chain, and the stationary
measure $\nu_F$ for $(Y_t)$ is the projection of $\mu_G$ to $F$.
Furthermore, consider the subsequence $(Z_t)$ of the chain $(Y_t)$
obtained by skipping those times when updates occurred outside of
$F$ in $(X_t)$. Namely, let $Z_t = Y_{K_t}$ where $K_t$ is the
$t$-th time that a block $\{v\}\cup F^c$ is updated in the chain
$(X_t)$. Clearly, $(Z_t)$ is a Markov chain on the space $\{-1,
1\}^F$, where at each time a uniform vertex $v$ from $F$ is selected
and updated according to the conditional Gibbs measure $\mu_G$ given
the spins on $F\setminus \{v\}$. The stationary measure for $(Z_t)$ is
also $\nu_F$.

Let $\cS_t = \sum_{v \in F}Z_t(v)$ be the sum of spins over $F$ in
the chain $(Z_t)$. It turns out that $\cS_t$ is a distinguishing
statistic and its analysis yields a lower bound on the
mixing time for chain $(Z_t)$. To this end, we need to estimate the
first two moments of $S_t$.
\begin{lemma}\label{lem-expectation-bound}
Let $(Z_t^{(+)})$ be an instance of the chain $(Z_t)$ started at the
all-plus configuration. Then its sum of spins satisfies that
$$\E_{+}(\cS_t) \geq |F| \Big(1 - \tfrac{1}{|F|}\Big)^t\,.$$
\end{lemma}
\begin{proof}
The proof follows essentially from a coupon collecting argument.
Let $(Z_t^*)$ be another instance of the chain $(Z_t)$, started from
$\nu_F$. It is obvious that we can construct a monotone coupling
between $(Z_t^{(+)})$ and $(Z_t^*)$ (namely, $Z_t^{(+)} \geq Z_t^*$
for all $t\in \N$) such that the vertices selected for updating in
both chains are always the same. Denote by $U[t]$ this (random)
sequence of vertices updated up to time $t$. Note that $Z_t^*$ has
law $\nu_F$, even if conditioned on the sequence $U[t]$. Recalling
that $Z_t^{(+)} \geq Z^*_t$ and $\E_\mu \sigma(v) = 0$, we obtain
that
\begin{equation*}
\E_{+}[Z_t^{(+)}(v) \mid v\in U[t]] \geq 0\,.
\end{equation*}
It is clear that $Z_t^{(+)}(v) = 1$ if $v \not\in U[t]$. Therefore,
\[\E_{+}[Z_t^{(+)}(v)] \geq \P(v\not\in U[t]) = (1 - \tfrac{1}{|F|})^t\,.\]
Summing over $v\in F$  concludes the proof.
\end{proof}
We next establish a contraction result for the chain $(Z_t)$.  We
need the GHS inequality of \cite{GHS} (see also \cite{Lebowitz} and
\cite{EM}). To state this inequality, we recall the definition of
the Ising model {\em with an external field}. Given a finite graph
$G=(V,E)$ with interaction strengths $J=\{J_{uv} \geq 0: uv\in E\}$
and external magnetic field $H=\{H_v: v\in V\}$,  the probability
for a configuration $\sigma \in \Omega = \{\pm1\}^V$
 is given by
 \begin{equation} \label{gibbsH}
 \mu_G^H(\sigma)  = \frac1{Z(J, H)} \exp\Big( \sum_{uv\in E}J_{uv}\sigma(u)\sigma(v) + \sum_{v \in V} H(v) \sigma(v)\Big)\,,
 \end{equation}
where $Z(J, H)$ is a normalizing constant.  Note that this specializes to (\ref{gibbs}) if $H \equiv 0$. When there is no ambiguity for the base graph, we sometimes drop the subscript $G$.
We can now state the

\medskip

\noindent{\bf GHS inequality~\cite{GHS}.} For a graph $G = (V, E)$,
let $\mu^H=\mu_G^H$ as above, and denote by $m_v(H) =
\E_{\mu^H}[\sigma(v)]$ the local magnetization at vertex $v$. If
$H_v \geq 0$ for all $v\in V$, then for any three vertices $u, v, w
\in V$ (not necessarily distinct),
\[\frac{\partial^2 m_v(H)}{\partial H_u \partial H_w} \leq 0\,.\]
The following is a consequence of the GHS inequality.
\begin{corollary}\label{cor-subaddtive}
For the Ising measure $\mu$ with no external field, we have
\begin{equation*}
\E_\mu[\sigma(u) \mid v_i = 1 \mbox{ for all }1\leq i \leq k] \leq \sum_{i=1}^k \E_\mu[\sigma(u) \mid v_i = 1]\,.
\end{equation*}
\end{corollary}
\begin{proof}
The function $f(H) = m_u(H)$ satisfies $f(0) = 0$.
 By the GHS inequality and Claim~\ref{claim-analysis} below, we
obtain that for all $H, H'\in \R_+^n$:
\begin{equation} \label{subadd}
m_u(H + H') \leq m_u(H) + m_u(H')\,.
\end{equation}
For $1 \leq i\leq k$ and $h\geq 0$, let $H_i^{h}$ be the external
field taking value $h$ on $v_i$ and vanishing on $V\setminus\{v_i\}$.
Applying the  inequality (\ref{subadd}) inductively, we deduce that
\[m_u\Big(\mbox{$\sum_i$} H_i^h\Big) \leq  \mbox{$\sum_i$} m_u(H_i^h)\,.\]
 Finally, let
$h\to \infty$ and observe that $m_u(H_i^h) \to
\E_{\mu}[\sigma(u)\mid \sigma(v_i) =1]$ and $m_u(\sum_i H_i^h) \to
\E_{\mu}[\sigma(u)\mid \sigma(v_i) =1 \mbox{ for all }1\leq i \leq
k]$.
\end{proof}
\begin{claim}\label{claim-analysis}
Write $\R_+ = [0, \infty)$ and let $f: \R_+^n \mapsto \R$ be a $C^2$-function such
that $\frac{\partial^2 f(x)}{\partial x_i\partial x_j} \leq 0$ for
all $x \in \R_+^n$ and $1\leq i,j\leq n$. Then for all $x, y\in
\R_+^n$,
\[f(x+y) - f(x) \leq f(y) - f(0)\,.\]
\end{claim}
\begin{proof}
Since all the second derivatives are non-positive,
$\frac{\partial f(x)}{\partial x_i}$ is decreasing in every
coordinate with $x$ for all $x\in \R_+^n$ and $i\leq n$. Hence,
$\frac{\partial f(x)}{\partial x_i}$ is decreasing in $\R_+^n$. Let
$$g_x(t) = \frac{d f(x + ty)}{d t} =\sum_i y_i \frac{\partial f(x)}{\partial x_i}(x+ty).$$ It  follows that $g_x(t)
\leq g_0(t)$ for all $x, y\in \R_+^n$. Integrating over $t\in [0,
1]$ yields the claim.
\end{proof}
\begin{lemma}\label{lem-contraction}
Suppose that $n\geq \mathrm{e}^{4}$. Let $(\tilde{Z}_t)$ be another
instance of the chain $(Z_t)$. Then for all starting states $z_0$
and $\tilde{z}_0$, there exists a coupling such that
\[\E_{z_0, \tilde{z}_0} \Big[\sum_{v\in F} |Z_t(v) - \tilde{Z}_t(v)|\Big] \leq \Big(1 - \frac{1}{2|F|}\Big)^t \sum_{v\in F} |z_0(v) - \tilde{z}_0(v)|\,.\]
\end{lemma}
\begin{proof}
Fix $\eta, \tilde{\eta} \in \{-1, 1\}^F$ such that $\eta$ and
$\tilde{\eta}$ differ only at the vertex $v$ and $\eta(v) = 1$. We
consider two chains $(Z_t)$ and $(\tilde{Z}_t)$ under monotone
coupling, started from $\eta$ and $\tilde{\eta}$ respectively.
Let $\eta_A$ be the restriction of $\eta$ to $A$ for $A \subset F$
(namely, $\eta_A \in  \{-1, 1\}^A$ and $\eta_A(v) = \eta(v)$ for all
$v\in A$), and write
\[\psi(u, \eta, \tilde{\eta}) = \E_\mu\big[\sigma(u) \mid \sigma_{F \setminus \{u\}} = \eta_{F \setminus \{u\}}\big] - \E_\mu\big[\sigma(u)\mid \sigma_{F \setminus \{u\}}= \tilde{\eta}_{F \setminus \{u\}}\big]\,.\] By the monotone property and symmetry of the Ising model,
\begin{align*}
\psi(u, \eta, \tilde{\eta}) &\leq \E_\mu[\sigma(u) \mid \sigma_{F\setminus \{u\}} = + ] - \E_\mu[\sigma(u)\mid \sigma_{F\setminus\{u\}} = -]\nonumber\\
 &=2 \E_\mu[\sigma(u) \mid \sigma_{F\setminus \{u\}} = + ] \,.
\end{align*}

By Corollary~\ref{cor-subaddtive},
\begin{align*}
\psi(u, \eta, \tilde{\eta}) \leq2\sum_{w\in F \setminus \{u\}}
\E_\mu[\sigma(u)\mid \sigma(w) =1] = 2 \sum_{w\in F\setminus \{u\}}
\Cov(\sigma(u), \sigma(w))\,.
\end{align*}
Recalling the non-negative correlations between the spins, we deduce
that under the monotone coupling
\begin{align*} \E_{\eta, \tilde{\eta}}  \Big[\,\frac{1}{2}\sum_{v\in F} |Z_1(v) - \tilde{Z}_1(v)|\Big] &= 1 - \frac{1}{|F|} +
\frac{1}{2|F|}\sum_{u \in F \setminus \{v\}} \psi(u, \eta, \tilde{\eta})\\
&\leq 1 - \frac{1}{|F|} + \frac{1}{|F|}\sum_{u \in F \setminus
\{v\}} \sum_{w\in F\setminus \{u\}} \Cov(\sigma(u), \sigma(w))\,.
\end{align*}
By Lemma~\ref{lem-F}, we get that for $n\geq \mathrm{e}^4$,
\begin{align*}
\E_{\eta, \tilde{\eta}}  \Big[\,\frac{1}{2}\sum_{v\in F} |Z_1(v) - \tilde{Z}_1(v)|\Big]\leq 1 - \frac{1}{|F|} + \frac{2}{|F|\log n} \leq 1 -
\frac{1}{2|F|}\,.
\end{align*}
Using the triangle inequality and recursion, we conclude the proof.
\end{proof}

From the contraction result, we can   derive the uniform variance bound on $\cS_t$. This type of argument appeared in \cite{LLP} (see Lemma 2.4) when $(Z_t)$ is a one dimensional chain. The  argument naturally extends to multi-dimensional case and we include the proof for completeness.
\begin{lemma}
Let $(Z_t)$ and $(\tilde{Z_t})$ be two instances of a Markov chain
taking values in $\R^n$. Assume that for some $\rho < 1$ and all
initial states $z_0$ and $ \tilde{z}_0$, there exists a coupling
satisfying \[\E_{z_0, \tilde{z}_0}\big[\mbox{$\sum_{i}$} |Z_t(i) -
\tilde{Z}_t (i)|\big] \leq \rho^t \mbox{$\sum_i$}|z_0(i) -
\tilde{z}_0(i)|\,,\] where we used the convention that $z(i)$ stands
for the $i$-th coordinate of $z$ for $z\in \R^n$. Furthermore,
suppose that $\sum_i|Z_t(i) - Z_{t-1}(i)| \leq R$ for all $t$. Then
for any $t\in \N$ and starting state $z \in \R^n$,
\[\var_{z} \big(\mbox{$\sum_i$}Z_t(i)\big) \leq \frac{2}{1 - \rho^2} R^2\,.\]
\end{lemma}
\begin{proof}
Let $Z_t$ and $Z'_t$ be two independent instances of the chain both
started from $z$. Defining $Q_t = \sum_{i}Z_t(i)$ and $Q'_t =
\sum_{i}Z'_t(i)$, we obtain that
\begin{align*}
\big|\E_{z}[Q_t\mid Z_1 = z_1] - \E_z[Q'_t |
Z'_t = z'_1]\big| &= \big|\E_{z_1}[Q_{t-1}] - \E_{z'_1} [Q'_{t-1}]\big|\\
&\leq\rho^{t-1}  \mbox{$\sum_i$}|z_1(i) - z'_1(i)| \leq
2\rho^{t-1}R\,,
\end{align*}
for all possible choices of $z_1$ and $z'_1$. It follows that for
any starting state $z$
\begin{align*}
\var_z(\E_z[Q_t\mid Z_1])=
\tfrac{1}{2}\E_z\big[\big(\E_{Z_1}[Q_{t-1}] -
\E_{Z'_1}[Q'_{t-1}]\big)^2\big]\leq 2(\rho^{t-1} R)^2.
\end{align*}
Therefore, by the total variance formula, we obtain that for all $z$
\begin{equation*}\var_z (Q_t)  = \var_z(\E_z[Q_t \mid Z_1]) + \E_z[\var_z(Q_t\mid
Z_1)]\leq 2(\rho^{t-1}R)^2 + \nu_{t-1}\,,
\end{equation*}
where $\nu_t \deq \max_z \var_z(Q_t)$. Thus $\nu_t \leq
2(\rho^{t-1}R)^2 + \nu_{t-1}$, whence
\[\nu_t \leq \sum_{i=1}^t (\nu_i - \nu_{i-1}) \leq \sum_{i=1}^{t} 2\rho^{2(t-1)}R^2 \leq \frac{2R^2}{1-\rho^2}\,,\]
completing the proof.
\end{proof}
Combining the above two lemmas gives the following variance bound (note that in our case $R=2$ and $\rho = 1 - \frac{1}{2|F|}$, so
$1-\rho^2  \ge \frac{1}{2|F|}$).
\begin{lemma}\label{lem-variance-bound}
For all $t$ and starting position $z$, we have  $\var_z (\cS_t)
\leq 16|F|$.
\end{lemma}
We can now derive a lower bound on the mixing time for the chain
$(Z_t)$.
\begin{lemma}\label{lem-mixing-Z}
The chain $(Z_t)$ has a mixing time
$\tmix^+ \geq \frac{1}{2} |F| \log |F| - 20|F|$.
\end{lemma}
\begin{proof}
Let $(Z_t^{(+)})$ be an instance of the dynamics $(Z_t)$
started from the all-plus configuration and let $Z^* \in \{-1, 1\}^F$ be distributed as $\nu_F$.
Write \[T_0 = \tfrac{1}{2}|F|\log |F| - 20|F|\,.\] It suffices to
prove that
\begin{equation}\label{eq-S-t-S}
d_{\mathrm{TV}}(\cS_{T_0}^{(+)}, \cS^*) \geq \tfrac{1}{4}\,,
\end{equation}
where $\cS_{T_0}^{(+)} = \sum_{v\in F}Z_{T_0}^{(+)}(v)$ as before and $\cS^* =
\sum_{v\in F}Z^*(v)$ be the sum of spins in stationary distribution.
 To this end, notice that by Lemmas
\ref{lem-expectation-bound} and \ref{lem-variance-bound}:
\[\E_+ (\cS_{T_0}^{(+)}) \geq \mathrm{e}^{20+o(1)}\sqrt{|F|}\, \mbox{ and } \var_+(\cS_{T_0}^{(+)}) \leq 16|F|\,.\]
An application of Chebyshev's inequality gives that for large enough $n$
\begin{equation}\label{eq-dist-S-T}
\P_+(\cS_{T_0}^{(+)} \leq \mathrm{e}^{10} \sqrt{|F|}) \leq
\frac{16|F|}{(\mathrm{e}^{20+o(1)}-\mathrm{e}^{10})\sqrt{|F|})^2}
\leq \frac{1}{4}\,.
\end{equation}
On the other hand, it is clear by symmetry that $\E_{\nu_F}\cS^* =
0$. Moreover, since Lemma \ref{lem-variance-bound} holds for all
$t$, taking $t\to\infty$ gives that $\var_{\nu_F}\cS^* \leq 16|F|$.
Applying Chebyshev's inequality again, we deduce that
\[\P_{\nu_F}(\cS^* \geq \mathrm{e}^{10} \sqrt{|F|}) \leq \frac{16|F|}{(\mathrm{e}^{10}\sqrt{|F|})^2}\leq \frac{1}{4} \,.\]
Combining the above inequality with \eqref{eq-dist-S-T} and the fact
that
\[d_{\mathrm{TV}}(\cS_{T_0}^{(+)}, \cS^*) \geq 1 - \P_+(\cS_{T_0}^{(+)} \leq \mathrm{e}^{10} \sqrt{|F|}) - \P_\mu(\cS^* \geq \mathrm{e}^{10} \sqrt{|F|})\,,\]
we conclude that \eqref{eq-dist-S-T} indeed holds (with room to spare), as required.
\end{proof}

We are now ready to derive Theorem \ref{mainthm-1}. Observe that the
dynamics $(Y_t)$ is a lazy version of the dynamics $(Z_t)$. Consider
an instance $(Y_t^{+})$ of the dynamics $(Y_t)$ started from the
all-plus configuration and let $Y^* \in \{-1, 1\}^F$ be distributed according to  the stationary
distribution $\nu_F$. Let $\cS_t^{(+)}$ and $\cS^*$ again be the sum
of spins over $F$, but with respect to the chain $(Y_t^{(+)})$ and
the variable $Y^*$ respectively.  Write
\[T = \frac{n}{|F|}\Big(\frac{1}{2}|F| \log |F| - 40|F|\Big)\,,\]
and let $N_T$ be the number of steps in $[1,T]$
where a block of the form $\{v\}\cup F$ is selected to update
in the chain $(Y_t^{(+)})$.
By Chebyshev's inequality,
\[\P(N_T \geq \tfrac{1}{2}|F| \log |F| - 20|F|) \leq \frac{T|F|/n}{(20|F|)^2} = o(1)\,.\]
Repeating the arguments in the proof of Lemma~\ref{lem-mixing-Z}, we   deduce that for all $t \leq  T_0=\frac{1}{2}|F| \log |F| - 20 |F|$, we have
\[\P_+(\cS_t^{(+)} \leq \mathrm{e}^{10} \sqrt{|F|}) \leq \tfrac{1}{4}\,.\]
Therefore
\begin{align*}\|\P_+(Y_T^{(+)} \in \cdot) - \nu_F\|_{\mathrm{TV}} &\geq 1 - \P(N_T \geq T_0) - \P_{\mu_Y}(\cS^* \geq \mathrm{e}^{10}\sqrt{|F|})\\
&\quad-\P_+\big(\cS_T^{(+)} \leq \mathrm{e}^{10}\sqrt{|F|}\mid N_T
\leq T_0\big)\,.
\end{align*}
 Altogether, we have that
\[\|\P_+(Y_T^{(+)} \in \cdot) - \nu_{F}\|_{\mathrm{TV}} \geq \tfrac{1}{2}+o(1)\geq \tfrac{1}{4}\,,\]
and hence that
\[\tmix^{+, Y}
\geq T \geq \tfrac{1+o(1)}{4} n \log n\,,\] where $\tmix^{+, Y}$
refers to the mixing time for chain $(Y_t^{(+)})$. Since the chain
$(Y_t)$ is a projection of the chain $(X_t)$, it follows
that the mixing time for the chain $(X_t)$ satisfies   $\tmix^{+,
X} \geq (1/4+o(1)) n\log n$. Combining this bound with Theorem
\ref{thm-PW-speed-up} (see the discussion following the statement of
the theorem), we conclude that the Glauber dynamics started with the
all-plus configuration has   mixing time $\tmix^+\geq (1/4+o(1)) n
\log n$. \qed
\begin{remark*}
The analysis naturally extends to the continuous-time Glauber
dynamics, where each site is associated with an independent Poisson
clock of unit rate determining the update times of this site as
above (note that the continuous dynamics is $|V|$ times faster than
the discrete dynamics). We can use similar arguments to these used above to handel the laziness in the transition from the chain $(Z_t)$ to the
chain $(Y_t)$. Namely, we could condition on the number of updates
up to time $t$ and then repeat the above arguments to establish that
$\tmix^+ \geq (1/4+o(1))\log n$ in the continuous-time case.
\end{remark*}

\begin{remark*}
We believe that Theorem \ref{mainthm-1} should have analogues (with $\tmix$ in place of $\tmix^+$) for the Ising model with arbitrary magnetic field, as well as for the Potts model and proper colorings. The first of these may be accessible to the methods of this paper, but the other two models need
new ideas.
\end{remark*}

\section*{Acknowledgments}
We thank Allan Sly and Asaf Nachmias for helpful comments.

\section*{Appendix: A simple proof for an almost optimal lower bound}

\noindent The current section is added on September 24, 2013.

\bigskip

We record here a simple proof found recently which gives that $$\inf_{G, J}t_{\mathrm{mix}}^+(G, J) \geq  n \log n/2- 3n \log\log n\,.$$

First of all, we can assume the spectral gap is larger than $1/(n\log n)$,
otherwise the lower bound holds since the mixing time is larger than the
inverse of the gap. Then we take two instances of Glauber dynamics where
one is started from all-plus configuration and one is from stationary
distribution $\mu$, and consider the monotone coupling between the two
chains. Let $\mathcal A$ be the random subset which has been updated by time
$t_n =  n \log n/2 - 3n \log\log n$ in the dynamics (same for both chains). It is clear that for any fixed
subset $A$, the distribution of the stationary chain at time $t_n$ remains
stationary under the conditioning $\mathcal A = A$. Denote by $S_A$ the sum of
spins over set $A$, and by $S$ the sum of spins over the whole graph. By
Dirichlet form, we get that $\var_\mu(S_A)\leq n \log n$ for any fixed subset $A$. Therefore, we see
$$\P_\mu(S_A \leq -10 \sqrt{n\log n}\mid \mathcal A = A) \leq 1/10$$ for any
fixed subset $A$. Averaging over the random set $\mathcal A$, we get that
$$\P_\mu(S_{\mathcal A} \leq -10 \sqrt{n\log n}) \leq 1/10\,.$$
Using the monotone coupling, we deduce that $$\P_+(S_{\mathcal A}
\leq -10 \sqrt{n \log n}) \leq 1/10\,.$$ In addition, an easy coupon
collecting argument gives that $$\P(|\mathcal A^c| \leq \sqrt{n}( \log n)^2)
\leq 1/10\,.$$ Altogether, we see that $\P_+(S\leq \sqrt{n} (\log n)^2/2) \leq
1/5$. Combined with the fact that $\var_\mu(S) \leq n\log n$, it follows that
the total variation distance between the all-plus chain at time $t_n$ and
the stationary distribution is at least 3/5, completing the proof of the
claim.

\end{document}